\documentclass[12pt,a4paper]{article}

\usepackage[T2A]{fontenc}
\usepackage[utf8]{inputenc}
\usepackage[russian,english]{babel}

\usepackage{geometry}
\usepackage{setspace}
\usepackage{microtype}

\usepackage{mathptmx}

\usepackage{amsmath,amssymb,amsfonts,amsthm}
\usepackage{mathtools}
\usepackage{bm}
\usepackage{mathrsfs}
\usepackage{dsfont}
\usepackage{bbm}
\usepackage{stmaryrd}
\usepackage{yhmath}
\usepackage{cases}
\usepackage{tensor}
\usepackage{enumitem}

\usepackage{graphicx}
\usepackage{booktabs}

\usepackage{hyperref}
\hypersetup{
  colorlinks=true,
  linkcolor=red,
  citecolor=blue,
  urlcolor=blue
}

\theoremstyle{plain}

\newtheorem{theorem}{Theorem}
\newtheorem{lemma}{Lemma}
\newtheorem{proposition}{Proposition}
\newtheorem{corollary}{Corollary}

\theoremstyle{definition}
\newtheorem{definition}{Definition}
\newtheorem{assumption}{Assumption}

\theoremstyle{remark}
\newtheorem{remark}{Remark}

\numberwithin{equation}{section}

\newcommand{\R}{\mathbb{R}}

\newcommand{\PP}{\mathbb{P}}

\newcommand{\E}{\mathbb{E}}
\newcommand{\Var}{\operatorname{Var}}

\newcommand{\Law}{\operatorname{Law}}

\newcommand{\affmark}[1]{\textsuperscript{#1}}
\newcommand{\corrauth}{\textsuperscript{*}}

\begin{document}

\selectlanguage{english}

\begin{center}
{\Large\bfseries 
Uniform Lorden-type bounds for overshoot moments for standard exponential families:
small drift and an exponential correction}
\vspace{0.8em}
{
Kalimulina~E.Yu.\corrauth\affmark{1,2}
\quad
Kelbert~M.Ya.\affmark{3}
}
\vspace{0.4em}
{\itshape 
\\
We dedicate this article to the memory of Galina A.~Zverkina}
\vspace{0.8em}
\end{center}

\begin{minipage}{0.92\textwidth}
\small
\affmark{1}
Institute for Information Transmission Problems of RAS, 
Bolshoy Karetny per., Moscow 127051, Russia

\affmark{2}
Moscow State University (MSU), 
Leninskie Gory, Moscow 119991, Russia

\affmark{3}\,
National Research University Higher School of Economics, Department of Statistics and Data Analysis, Laboratory of Stochastic Analysis and Its Applications, Moscow, Russian Federation
\\[0.4em]
\corrauth\ Corresponding author: \texttt{elmira.yu.k@gmail.com}
\end{minipage}

\vspace{1.2em}

\noindent\textbf{Abstract.}
We study the overshoot \(R_b=S_{\tau(b)}-b\) of a random walk with independent identically
distributed increments from a standardised one-parameter exponential family, with primary
emphasis on the small-drift regime
\(\theta\downarrow0\).
Unlike the classical renewal-process setting with nonnegative increments,
we allow sign-changing increments and assume only a positive drift \(\mu_\theta>0\).
For each \(k\in\mathbb N\) we obtain Lorden-type moment bounds, uniform in the barrier \(b\),
for \(\E_\theta[R_b^k]\) with an explicit remainder term decaying exponentially in \(b\).
The proof reduces the problem to the renewal process of strict ascending ladder
heights and combines a simple bound for the limiting overshoot moments with a uniform exponential
estimate for the rate of convergence of the distribution functions of \(R_b\) to the limiting random variable \(R_\infty\)
as \(b\to\infty\), uniformly in \(\theta\in[0,\theta^\ast]\).
As a consequence, the classical constant \((k+2)/(k+1)\) arising in residual-life bounds
improves to \(C_k=1\) for sufficiently large \(b\) at fixed \(\theta\), and also
uniformly over all \(b\ge0\) in the small-drift regime.
Counterexamples are provided showing that the stronger inequality with \(k\mu_\theta\) in the denominator
cannot hold uniformly in \((b,\theta)\).
Finally, the exponential CDF estimate is interpreted in terms of optimal transport:
we obtain exponential convergence in the metric \(W_1\), a quantile coupling with
\(\E|\widetilde R_b-\widetilde R_\infty|=O(e^{-rb})\), error bounds for Lipschitz functionals
and a total-variation bound for smoothed distributions.



\vspace{0.4em}

\vspace{1.2em}

\selectlanguage{english}

\vspace{1.4em}

\selectlanguage{russian}


\section{Introduction}\label{sec:intro}

Let $(X_n)_{n\ge 1}$ be a sequence of independent identically distributed random variables (i.i.d.)
and let
\[
S_0:=0,\qquad S_n:=\sum_{i=1}^n X_i,\qquad n\ge 1.
\]
For $b>0$ set
\begin{equation}\label{eq:tau_def}
\tau(b):=\inf\{n\ge 1:\ S_n>b\},
\end{equation}
and define the \emph{overshoot}
\begin{equation}\label{eq:overshoot_def}
R_b:=S_{\tau(b)}-b.
\end{equation}

In the case of nonnegative increments, the following bound for the mean overshoot,
uniform in the level $b$, is known.

\begin{theorem}[Lorden's inequality {\rm(Lorden, 1970)}\cite{Lorden1970}]\label{thm:lorden1970}
Assume that $X_1,X_2,\dots$ are nonnegative a.\,s.\ and
\[
0<\mathbb{E}[X_1]<\infty,\qquad \mathbb{E}[X_1^2]<\infty.
\]
Then, in the notation of \eqref{eq:tau_def}--\eqref{eq:overshoot_def}, for every $b\ge 0$
the inequality
\begin{equation}\label{eq:lorden_ineq}
\mathbb{E}[R_b]\ \le\ \frac{\mathbb{E}[X_1^2]}{\mathbb{E}[X_1]}.
\end{equation}
holds.
\end{theorem}

Inequality \eqref{eq:lorden_ineq} was proved in \cite{Lorden1970}; for alternative proofs
and some of its generalisations, see, for example, \cite{CarlssonNerman1986,Chang1994,Spouge2007,Svensson02,Janson83}.

We next consider the standard one-parameter exponential family
$F_{\theta}(d x)=e^{\theta x-\psi(\theta)}F_0(d x)$,
where $E_0[X]=0$,
allow sign-changing increments (in contrast to the assumptions of Theorem~\ref{thm:lorden1970}),
and assume only a positive drift
\[
\mu_\theta:=\mathbb{E}_\theta[X_1]>0.
\]
In this case, a natural tool is the consideration of the renewal process associated with strict
ascending ladder epochs~\cite{Wong2025,Lotov2002}.

\begin{definition}[Strict ascending ladder epochs and heights \cite{Feller71}]\label{def:ladder_epochs}
Set $T_+^{(0)}:=0$ and for $n\ge1$ define recursively
\[
T_+^{(n)}:=\inf\bigl\{m>T_+^{(n-1)}:\ S_m>S_{T_+^{(n-1)}}\bigr\}.
\]
Set $T_+:=T_+^{(1)}$ and $H_n:=S_{T_+^{(n)}}$, $n\ge0$.
\end{definition}

It follows from Definition~\ref{def:ladder_epochs} that
\[
T_+=\inf\{n\ge1:\ S_n>0\}.
\]

\begin{lemma}[\cite{Wong2025}]\label{lem:tau_is_ladder}
For each $b>0$ there exists an integer-valued random variable $v(b)\ge 1,$
such that
\[
T_+^{(v(b))}=\tau(b),
\]
where $~~~ v(b):=\inf\{n\ge1:\ H_n>b\}$.
\end{lemma}

\begin{definition}[Renewal function for ladder heights]\label{def:renewal_function}
For $\theta\in\Theta$ define the renewal function \cite{Borovkov2020}
\begin{equation}\label{eq:renewal_function}
U_\theta^+(x):=\sum_{n=0}^{\infty}\mathbb{P}_\theta\!\left(H_n\le x\right),\qquad x\ge0.
\end{equation}
Let $U_\theta^+(\mathrm{d}t)$ denote the measure on $[0,\infty)$ determined by the increments
$U_\theta^+((a,b])=U_\theta^+(b)-U_\theta^+(a)$.
\end{definition}

The key identity is the renewal equation for the tail of the overshoot
(see~\cite{Wong2025}):
\begin{equation}\label{eq:renewal_eq_tail}
\mathbb{P}_\theta(R_b>y)
=\int_{[0,b]}\mathbb{P}_\theta\!\bigl(S_{T_+}>b+y-t\bigr)\,U_\theta^+(\mathrm{d}t),
\qquad b>0,\ y>0.
\end{equation}

The following statement is obtained by a standard renewal-theoretic argument from the definition
\eqref{eq:renewal_function} and equation \eqref{eq:renewal_eq_tail} ~\cite{Wong2025}.

\begin{proposition}[Limiting distribution and limiting overshoot moments]\label{prop:limit_moment}
Suppose that $\mu_\theta=\mathbb{E}_\theta[X_1]>0$, and that the increment distribution satisfies
the standard non-lattice assumptions ensuring the applicability of the key renewal theorem
to the strict ascending ladder-height process \cite{Smith58,Cox62}.
Then there exists a random variable $R_\infty$ such that
\[
R_b \Rightarrow R_\infty,\qquad b\to\infty.
\]
Moreover, for each $k\in\mathbb{N}$ such that $\mathbb{E}_\theta\!\left[S_{T_+}^{\,k+1}\right]<\infty$,
the limiting formula \cite{Wong2025}
\begin{equation}\label{eq:limit_moment}
\lim_{b\to\infty}\mathbb{E}_\theta\!\left[R_b^{\,k}\right]
=
\frac{\mathbb{E}_\theta\!\left[S_{T_+}^{\,k+1}\right]}{(k+1)\,\mathbb{E}_\theta\!\left[S_{T_+}\right]}.
\end{equation}
holds.
\end{proposition}

We note that the statement of Proposition~\ref{prop:limit_moment} is valid for a broad class of
non-lattice distributions with positive drift.
We now fix a parametric class of increment distributions and consider small-drift regimes
in which Lorden-type bounds for overshoot moments, uniform in $b$, are required.

\begin{definition}[Standard exponential family]\label{def:standard_exp_family}
Let $F_0$ be a probability measure on $\mathbb{R}$, and let $\Theta=[0,\theta^*)$, $\theta^*\in [0,\infty]$
be a half-interval containing $0$
 such that
\[
\int_{\mathbb{R}} e^{\theta x}\,F_0(\mathrm{d}x)<\infty,\qquad \theta\in\Theta.
\]
Define
\[
\psi(\theta):=\ln\int_{\mathbb{R}} e^{\theta x}\,F_0(\mathrm{d}x),\qquad \theta\in\Theta,
\]
and the family of probability measures $\{F_\theta:\theta\in\Theta\}$ by
\begin{equation}\label{eq:exp_family}
F_\theta(\mathrm{d}x)=e^{\theta x-\psi(\theta)}\,F_0(\mathrm{d}x),\qquad \theta\in\Theta.
\end{equation}
The family $\{F_\theta\}$ is called \emph{standard} if the normalisation
\[
\mathbb{E}_0[X]=0,\qquad \mathrm{Var}_0(X)=1 \quad (X\sim F_0),
\]
holds and $F_0$ is \emph{strongly non-lattice}, i.e.
\[
\limsup_{|t|\to\infty}|\chi(t)|<1,
\qquad
\chi(t):=\int_{\mathbb{R}} e^{itx}\,F_0(\mathrm{d}x).
\]
\end{definition}




\begin{assumption}[Increment model]\label{ass:increments_exp_family}
The random variables $X_1,X_2,\dots$ are independent and identically distributed, and
$X_1\sim F_\theta$ for some $\theta\in\Theta$, where $\{F_\theta\}$ is a standard
exponential family in the sense of Definition~\ref{def:standard_exp_family}.
\end{assumption}

\begin{remark}[Non-lattice property and applicability of the key renewal theorem]
Since $F_\theta(\mathrm{d}x)=e^{\theta x-\psi(\theta)}F_0(\mathrm{d}x)$,
the measures $F_\theta$ and $F_0$ are equivalent and have the same support.
Consequently, if $F_0$ is (strongly) non-lattice, then $F_\theta$ is non-lattice for all $\theta\in\Theta$.
In particular, the key renewal theorem applies to the strict ascending ladder-height process
for each fixed $\theta$, provided that $\E_\theta[S_{T_+}]<\infty$.
\end{remark}

For $\theta\in\Theta$ let $\mathbb{P}_\theta$, $\mathbb{E}_\theta$ and
$\mathrm{Var}_\theta$ denote the probability, expectation and variance, respectively,
under $X_1\sim F_\theta$. Then
\begin{equation}\label{eq:mu_theta}
\mu_\theta:=\mathbb{E}_\theta[X_1]=\psi'(\theta)>0,\qquad
\psi''(\theta)=\mathrm{Var}_\theta(X_1).
\end{equation}

We note that in \eqref{eq:limit_moment} the factor $(k+1)$ arises naturally in the denominator.

To conclude the survey part, we also recall the known moment bound for the renewal process \cite{Chang1994,Sugakova2007}.
Assume that $X_1,X_2,\dots$ are nonnegative a.\,s.\ and set
\[
N(b):=\inf\{n\ge1:\ S_n>b\},\qquad \zeta_b:=S_{N(b)}-b,\qquad b\ge0.
\]
Then $(S_n)$ is a (discrete) renewal process with arrivals $X_i$,
and the quantity $\zeta_b$ is the residual lifetime. In this case $\zeta_b=R_b$.

It is known (cf.\ \cite{Chang1994}; also formula \textup{(11)} in \cite{Sugakova2007}) that for $k>0$, if
$\E[X_1]\in(0,\infty)$ and $\E[X_1^{k+1}]<\infty$, then the inequality
\begin{equation}\label{eq:known_sugakova_chang}
\E[\zeta_b^{\,k}]
\ \le\
\frac{k+2}{k+1}\,\frac{\E[X_1^{k+1}]}{\E[X_1]},
\qquad b\ge0.
\end{equation}
holds.
(For $k=1$ there is a stronger canonical Lorden bound without the factor $\frac{3}{2}$:
$\E[\zeta_b]\le \E[X_1^2]/\E[X_1]$.)

The principal interest in problems of applied probability is often connected with the mean overshoot
(the case $k=1$). Nevertheless, in what follows we work with moments of arbitrary order
$k\in\mathbb N$.

\textbf{Related literature (small drift and ladder quantities).}
Bounds for the overshoot and for ladder functionals of a random walk in regimes of
\emph{positive but small drift} were also discussed in the work of V.\,I.~Lotov~\cite{Lotov2002},
where upper bounds were obtained for the tail of the first nonnegative partial sum and for overshoot moments above
an arbitrary level under the assumption that the increment mean is positive and small,
and a bound for $\E[T_+]$ was also given (cf.\ the abstract and the references in~\cite{Lotov2002}).
We note that in~\cite{Lotov2002} Mogulskii's absolute bounds for moment functionals
of boundary problems~\cite{Mogulskii1973} are explicitly mentioned among the basic tools.

The aim of the present paper is to obtain \emph{bounds uniform in the level $b$} of Lorden type
for the overshoot moments $\E_\theta[R_b^{\,k}]$ in small-drift regimes $\theta\downarrow 0$
within the model of a standard exponential family.
In particular, we are interested in the possibility of replacing the factor $\frac{k+2}{k+1}$ in the classical
moment bound for renewal processes by the constant $C_k=1$:
to show that for a sufficiently large barrier $b$ (and, in the small-drift regime, even uniformly over all $b\ge0$)
one has
\[
\E_\theta[R_b^{\,k}] \ \le\ \frac{\E_\theta[(X_1^+)^{k+1}]}{\mu_\theta}.
\]

Counterexamples are given in Appendix~A which
concern the
 \emph{strengthened} form of the bound with the factor $k\mu_\theta,~ k>1$ in the denominator
and show that, in general, such a bound cannot hold simultaneously for all $\theta$ and all $b\ge 0$.

Alongside the theoretical problem of obtaining Lorden-type moment bounds, uniform in the level $b$,
for the overshoot $R_b$, we pursue an applied aim: to provide explicit error control in threshold-stopping
problems and regenerative models (queueing theory, reliability, etc.), since from
the identity $S_{\tau(b)}=b+R_b$ and Wald's identity (whenever $\E_\theta[\tau(b)]<\infty$) it follows that
$\E_\theta[\tau(b)]=(b+\E_\theta[R_b])/\mu_\theta$ when $\mu_\theta>0$.
Further, the exponential rate of convergence $\Law(R_b)\to\Law(R_\infty)$ is interpreted via
optimal transport/coupling (cf.\ Section~\ref{sec:coupling_viewpoint}), which yields the bound
$W_1(R_b,R_\infty)=O(e^{-rb})$ and the existence of a joint distribution
$(\widetilde R_b,\widetilde R_\infty)$ with $\E|\widetilde R_b-\widetilde R_\infty|=O(e^{-rb})$, and hence
explicit bounds for the error of replacing $R_b$ by $R_\infty$ in mean values of Lipschitz functionals.

\subsection*{Main result of the paper}

Fix $k\in\mathbb N$ and denote $x^+:=\max\{x,0\}$.
In Section~\ref{sec:main_results} it is shown that, within the model of a standard exponential family
(Definition~\ref{def:standard_exp_family}, Assumption~\ref{ass:increments_exp_family}),
it is possible to improve the classical moment bound of the form \eqref{eq:known_sugakova_chang}
to the constant $C_k=1$ in the regimes of a large barrier and small drift.

\subsubsection*{What is new in this paper}

Unlike the classical uniform bounds for the residual lifetime in a pure renewal process,
we consider sign-changing increments and use a reduction to renewal over strict ascending
ladder heights. Our contribution consists of the following:

\begin{itemize}
\item we derive an explicit bound for the overshoot moments $\E_\theta[R_b^{\,k}]$ with an exponentially small in $b$
correction, uniform in $\theta\in(0,\theta^\ast]$;
\item from this bound we obtain an explicit threshold $b_0(\theta,k)$ from which onward the constant improves to $C_k=1$;
\item we show that for sufficiently small drift (small $\theta$) the improvement $C_k=1$ becomes uniform
over all $b\ge0$.
\end{itemize}

More precisely, the main results reduce to the following three assertions.

\begin{enumerate}
\item \textbf{Bound with an exponential correction in the level $b$.}
There exist constants $C>0$, $r>0$ and $\theta^\ast>0$ such that for all
$\theta\in(0,\theta^\ast]$ and $b>0$ the bound (see Theorem~\ref{thm:exp_remainder_ru})
\[
\E_\theta[R_b^{\,k}]
\ \le\
\frac{1}{k+1}\,\frac{\E_\theta[(X_1^+)^{k+1}]}{\mu_\theta}
\ +\
C\,\frac{k\Gamma(k)}{r^k}\,e^{-rb},
\qquad
\Gamma(k)=\int_0^\infty z^{k-1}e^{-z}\,dz.
\]
holds.

\item \textbf{Improvement to $C_k=1$ for sufficiently large $b$.}
For fixed $\theta\in(0,\theta^\ast]$ and $k\in\mathbb N$ there exists a threshold $b_0(\theta,k)\ge0$
such that for all $b\ge b_0(\theta,k)$ one has (see Corollary~\ref{cor:Ck1_large_b_ru})
\[
\E_\theta[R_b^{\,k}]
\ \le\
\frac{\E_\theta[(X_1^+)^{k+1}]}{\mu_\theta}.
\]

\item \textbf{Small drift: improvement to $C_k=1$ uniformly over all $b\ge0$.}
There exists $\theta_k\in(0,\theta^\ast]$ such that for all $\theta\in(0,\theta_k]$ and all $b\ge0$
we have (see Corollary~\ref{cor:Ck1_small_drift_ru})
\[
\E_\theta[R_b^{\,k}]
\ \le\
\frac{\E_\theta[(X_1^+)^{k+1}]}{\mu_\theta}.
\]
\end{enumerate}


\subsection*{Structure of the paper}

The paper is organised as follows.
In Section~\ref{sec:intro} we introduce the notation and recall the standard reduction of the overshoot
problem to the renewal process of strict ascending ladder heights, in particular deriving
the renewal equation \eqref{eq:renewal_eq_tail} for the tail $\PP_\theta(R_b>y)$ through the renewal
measure $U_\theta^+$.

In Section~\ref{sec:main_results} we prove the bound \eqref{eq:exp_remainder_ru} on the basis of
the uniform exponential rate of convergence of the overshoot in the level $b$
(Proposition~\ref{prop:wong_prop11_ru}), and we also derive the improvement of the constant to $C_k=1$
in the regimes of a large barrier and small drift (Corollaries~\ref{cor:Ck1_large_b_ru}
and~\ref{cor:Ck1_small_drift_ru}).

In Section~\ref{sec:coupling_viewpoint} an applied interpretation of the resulting exponential CDF estimate is given in terms of the Wasserstein distance, coupling and smoothed total variation, leading to explicit approximation-error bounds for Lipschitz functionals and to an illustrative example of a threshold policy (Subsection~\ref{subsec:example_threshold_policy}).

 The appendices collect technical details and
present counterexamples showing that the strengthening of the bound to the form \eqref{eq:target-ru}
with $k\mu_\theta$ in the denominator cannot be required simultaneously for all parameter values $\theta$
and all levels $b\ge0$.
\section{Overshoot moment bounds: statements and proofs of the main results}\label{sec:main_results}

This section contains proofs of the results stated in the introduction.
The key external tool is the exponential rate of convergence, uniform in $\theta$,
of the overshoot in the level $b$ (Proposition~\ref{prop:wong_prop11_ru}; cf.\ \cite{Wong2025}).
We first prove the bound with an exponential correction (Theorem~\ref{thm:exp_remainder_ru}),
and then derive from it the improvement of the constant in the Lorden-type moment bound to $C_k=1$
(Corollaries~\ref{cor:Ck1_large_b_ru} and~\ref{cor:Ck1_small_drift_ru}).

\subsection{The classical moment bound and its improvement in the large-barrier regime}
\label{subsec:lorden_core}

Throughout this subsection we use the notation of Section~\ref{sec:intro}.
In particular, $\tau(b)$ and $R_b$ are defined in \eqref{eq:tau_def}--\eqref{eq:overshoot_def},
while $T_+$ and $H_n=S_{T_+^{(n)}}$ are defined in Definition~\ref{def:ladder_epochs}.
We also assume that Definition~\ref{def:standard_exp_family} and
Assumption~\ref{ass:increments_exp_family} hold; for $\theta>0$ we have $\mu_\theta>0$.

\textbf{Basic (known) bound with constant \texorpdfstring{$\frac{k+2}{k+1}$}{(k+2)/(k+1)}}. Applying \eqref{eq:known_sugakova_chang} to the renewal process of strict ascending ladder
heights $(H_n)$ and then estimating from above the moments of $S_{T_+}$ by $(X_1^+)^{k+1}$ and $\mu_\theta$
(cf.\ \eqref{eq:mu_theta}), one obtains a bound of the form
\[
\sup_{b\ge 0}\E_\theta[R_b^{\,k}]
\ \le\
\frac{k+2}{k+1}\,\frac{\E_\theta[(X_1^+)^{k+1}]}{\mu_\theta},
\]
which, in essentially a similar form,
  is already present in the literature \cite{Sugakova2007,Chang1994}.

We are interested in how to \emph{improve the constant} to $C_k=1$ in regimes
natural for the problem "small drift + large barrier"\footnote{In works such as \cite{Lotov2002,Mogulskii1973}, ladder functionals are also analysed under a small-drift assumption. The distinction of our work lies in the improvement of the constant to $C_k=1$ via the exponential rate of convergence.}.


We note that the counterexamples from Appendix~A show the impossibility of
a further strengthening to the form \eqref{eq:target-ru} with the factor $k\mu_\theta$ in the denominator.

\subsubsection*{Key known fact: uniform exponential rate (see \cite{Wong2025,Chang1992})}

\begin{proposition}[Uniform exponential convergence of the overshoot \cite{Wong2025}%
]
\label{prop:wong_prop11_ru}
Let $\{F_\theta:\theta\ge 0\}$ be a standard exponential family and let $F_0$ be strongly non-lattice.
Then there exist constants $C>0$, $r>0$ and $\theta^\ast>0$ such that for all $b>0$, $y>0$
\begin{equation}\label{eq:wong_cdf_rate_ru}
\sup_{\theta\in[0,\theta^\ast]}
\bigl|\PP_\theta(R_b\le y)-\PP_\theta(R_\infty\le y)\bigr|
\ \le\ C e^{-r(b+y)}.
\end{equation}
Consequently, for any $k>0$ and all $b>0$ we have
\begin{equation}\label{eq:wong_moment_rate_ru}
\sup_{\theta\in[0,\theta^\ast]}
\bigl|\E_\theta[R_b^{\,k}]-\E_\theta[R_\infty^{\,k}]\bigr|
\ \le\
C\,\frac{k\Gamma(k)}{r^k}\,e^{-rb},
\qquad \Gamma(k)=\int_0^\infty z^{k-1}e^{-z}\,dz.
\end{equation}
\end{proposition}

\begin{remark}
Estimate \eqref{eq:wong_moment_rate_ru} follows from \eqref{eq:wong_cdf_rate_ru}
by tail integration: $\E[Z^k]=k\int\limits_0^\infty y^{k-1}\PP(Z>y)\,dy$.
\end{remark}

\subsubsection*{The bound \texorpdfstring{$\E_\theta[R_b^k]$}{E[R_b^k]} with an exponential correction}

\begin{lemma}\label{lem:ETplus_exp_family}
Assume that Definition~\ref{def:standard_exp_family} and Assumption~\ref{ass:increments_exp_family} hold.
Then for every $\theta\in\Theta$ such that $\theta>0$, we have $\E_\theta[T_+]<\infty$, where $T_+:=\inf\{n\ge1:\ S_n>0\}$.
\end{lemma}

\begin{proof}
Set $\eta:=\theta$. Then from \eqref{eq:exp_family} we obtain
\[
\E_\theta\bigl[e^{-\eta X_1}\bigr]
=\int e^{-\theta x}\,e^{\theta x-\psi(\theta)}\,F_0(dx)
=e^{-\psi(\theta)}.
\]
The function $\psi$ is strictly convex, with $\psi(0)=0$ and $\psi'(0)=\E_0[X_1]=0$,
therefore $\psi(\theta)>0$ for $\theta>0$, and hence $\E_\theta[e^{-\theta X_1}]<1$.
Further, on the event $\{T_+>n\}$ we have $S_n\le0$, so that $e^{-\theta S_n}\ge 1$, and
\[
\PP_\theta(T_+>n)\le \E_\theta[e^{-\theta S_n}]
=\bigl(\E_\theta[e^{-\theta X_1}]\bigr)^n
=e^{-n\psi(\theta)}.
\]
Summing, we obtain $\E_\theta[T_+]<\infty$.
\end{proof}

\begin{lemma}\label{lem:ladder_ratio_bound_ru}
Let $\theta>0$ and $k\in\mathbb N$. Then
\begin{equation}\label{eq:rinfty_bound_ru}
\E_\theta[R_\infty^{\,k}]
\ \le\
\frac{1}{k+1}\,\frac{\E_\theta[(X_1^+)^{k+1}]}{\mu_\theta}.
\end{equation}
\end{lemma}

\begin{proof}
By Proposition~\ref{prop:limit_moment} (cf.\ \eqref{eq:limit_moment}) \cite{Wong2025}, we have
\[
\E_\theta[R_\infty^{\,k}]
=
\frac{\E_\theta[S_{T_+}^{\,k+1}]}{(k+1)\E_\theta[S_{T_+}]}.
\]
Since $S_{T_+-1}\le 0<S_{T_+}$, it follows that $0<S_{T_+}\le X_{T_+}^+$, whence
\[
S_{T_+}^{\,k+1}\le (X_{T_+}^+)^{k+1}\le \sum_{i=1}^{T_+}(X_i^+)^{k+1}.
\]
Taking expectations and applying Wald's identity (which is applicable because for $\theta>0$
we have $\E_\theta[T_+]<\infty$, and $\E_\theta[(X_1^+)^{k+1}]<\infty$ in the exponential family), we obtain
\[
\E_\theta[S_{T_+}^{\,k+1}]
\le
\E_\theta\!\Big[\sum_{i=1}^{T_+}(X_i^+)^{k+1}\Big]
=
\E_\theta[T_+]\,\E_\theta[(X_1^+)^{k+1}].
\]
On the other hand, by Wald's identity
\[
\E_\theta[S_{T_+}]
=
\E_\theta\!\Big[\sum_{i=1}^{T_+}X_i\Big]
=
\E_\theta[T_+]\,\mu_\theta.
\]
Dividing one by the other, we obtain
\eqref{eq:rinfty_bound_ru}.
\end{proof}

\begin{theorem}[Overshoot moment bound with an exponential correction]
\label{thm:exp_remainder_ru}
Assume that Definition~\ref{def:standard_exp_family} and Assumption~\ref{ass:increments_exp_family} hold.
Then there exist $C>0$, $r>0$ and $\theta^\ast>0$ such that for any $k\in\mathbb N$, any $b>0$ and
any $\theta\in(0,\theta^\ast]$,
\begin{equation}\label{eq:exp_remainder_ru}
\E_\theta[R_b^{\,k}]
\ \le\
\frac{1}{k+1}\,\frac{\E_\theta[(X_1^+)^{k+1}]}{\mu_\theta}
\ +\
C\,\frac{k\Gamma(k)}{r^k}\,e^{-rb}.
\end{equation}
\end{theorem}

\begin{proof}
From \eqref{eq:wong_moment_rate_ru} we obtain
\[
\E_\theta[R_b^{\,k}]
\le
\E_\theta[R_\infty^{\,k}]
+
C\,\frac{k\Gamma(k)}{r^k}\,e^{-rb}.
\]
It remains to apply Lemma~\ref{lem:ladder_ratio_bound_ru}, which yields \eqref{eq:exp_remainder_ru}.
\end{proof}

\subsubsection*{Main improvement: the constant \texorpdfstring{$C_k=1$}{C_k=1} for a large barrier and for small drift}

\begin{corollary}[The constant $C_k=1$ for sufficiently large $b$]
\label{cor:Ck1_large_b_ru}
Assume that the conditions of Theorem~\ref{thm:exp_remainder_ru} hold, and let $k\in\mathbb N$, $\theta\in(0,\theta^\ast]$.
Set
\[
A_{\theta,k}:=\frac{\E_\theta[(X_1^+)^{k+1}]}{\mu_\theta},
\qquad
B_k:=C\,\frac{k\Gamma(k)}{r^k},
\qquad
\ln_+(x):=\max\{\ln x,0\}.
\]
Define the threshold
\begin{equation}\label{eq:b0_def_ru}
b_0(\theta,k):=\frac{1}{r}\ln_+\!\left(\frac{(k+1)B_k}{k\,A_{\theta,k}}\right).
\end{equation}
Then for all $b\ge b_0(\theta,k)$ the improved inequality
\begin{equation}\label{eq:Ck1_bound_ru}
\E_\theta[R_b^{\,k}]
\ \le\
\frac{\E_\theta[(X_1^+)^{k+1}]}{\mu_\theta}
\;=\;A_{\theta,k}.
\end{equation}
holds.
\end{corollary}

\begin{proof}
From \eqref{eq:exp_remainder_ru} we have
\[
\E_\theta[R_b^{\,k}] \le \frac{A_{\theta,k}}{k+1}+B_k e^{-rb}.
\]
If $b\ge b_0(\theta,k)$, then by \eqref{eq:b0_def_ru}
$B_k e^{-rb}\le \frac{k}{k+1}A_{\theta,k}$, and hence
\[
\E_\theta[R_b^{\,k}]
\le
\frac{A_{\theta,k}}{k+1}+\frac{k}{k+1}A_{\theta,k}
=
A_{\theta,k}.
\]
\end{proof}

\begin{corollary}[Small drift: $C_k=1$ uniformly in $b$]
\label{cor:Ck1_small_drift_ru}
Assume that the conditions of Theorem~\ref{thm:exp_remainder_ru} hold and let $k\in\mathbb N$.
Then there exists $\theta_k\in(0,\theta^\ast]$ such that for all $\theta\in(0,\theta_k]$ and all $b\ge 0$
\eqref{eq:Ck1_bound_ru} holds.
\end{corollary}

\begin{proof}
By the definition of a standard exponential family, $\mu_\theta=\psi'(\theta)\to 0$ as $\theta\downarrow 0$,
while $\E_\theta[(X_1^+)^{k+1}]\to \E_0[(X_1^+)^{k+1}]>0$ (the inequality is strict because $\Var_0(X_1)=1$).
Consequently, $A_{\theta,k}=\E_\theta[(X_1^+)^{k+1}]/\mu_\theta\to\infty$ as $\theta\downarrow 0$.
Choose $\theta_k$ so that for all $\theta\in(0,\theta_k]$ one has
$B_k \le \frac{k}{k+1}A_{\theta,k}$.
Then from \eqref{eq:exp_remainder_ru}, for $b\ge 0$ (since $e^{-rb}\le 1$), we obtain
\[
\E_\theta[R_b^{\,k}] \le \frac{A_{\theta,k}}{k+1}+B_k \le A_{\theta,k},
\]
which yields \eqref{eq:Ck1_bound_ru}.
\end{proof}

\section{Overshoot bounds in coupling problems and estimates of the rate of convergence:
Wasserstein distance, smoothed total variation and the error of replacing $R_b$ by $R_\infty$}
\label{sec:coupling_viewpoint}

The exponential estimate for the rate of convergence of the overshoot distribution in the level $b$,
see~\eqref{eq:wong_cdf_rate_ru} (Proposition~\ref{prop:wong_prop11_ru}),
is the key external tool in the proofs of the results of
Section~\ref{sec:main_results}.
In the present section we record several \emph{standard} consequences of this CDF estimate,
which are convenient for applied approximation problems and for constructing couplings in
regenerative models and renewal processes \cite{Lindvall02}, in particular in queueing theory and reliability \cite{Zverkina20,KalimulinaZverkina21}.
Namely, \eqref{eq:wong_cdf_rate_ru} yields:
(i) an exponential rate of convergence in the Wasserstein distance of order $1$ and, as a consequence,
the existence of a quantile coupling $(\widetilde R_b,\widetilde R_\infty)$ with
$\E_\theta|\widetilde R_b-\widetilde R_\infty|=O(e^{-rb})$;
(ii) universal bounds for the error of replacing $R_b$ by $R_\infty$ in expectations of Lipschitz functionals;
(iii) when needed, control in total variation after the standard smoothing of the distributions.
These consequences are presented for the convenience of subsequent applications and do not claim novelty.

We next use these bounds to obtain explicit a priori bounds for the error of replacing $R_b$ by $R_\infty$
in typical threshold approximations (see~\S\ref{subsec:example_threshold_policy}).

\subsubsection*{Notation, the basic CDF estimate and preliminary results}

For $y\ge 0$ set
\[
F_{b,\theta}(y):=\PP_\theta(R_b\le y),\qquad
F_{\infty,\theta}(y):=\PP_\theta(R_\infty\le y).
\]
According to Proposition~\ref{prop:wong_prop11_ru}, there exist constants $C>0$, $r>0$ and $\theta^\ast>0$
such that for all $b>0$ and $y>0$
\begin{equation}\label{eq:cdf_exp_rate_recall}
\sup_{\theta\in[0,\theta^\ast]}
\bigl|F_{b,\theta}(y)-F_{\infty,\theta}(y)\bigr|
\ \le\ C e^{-r(b+y)}.
\end{equation}

\subsubsection*{The Wasserstein distance of order $1$ and the one-dimensional formula}

\begin{definition}[The Wasserstein distance $W_1$]
Let $X,Y$ be random variables on $\mathbb{R}$ with finite first moments.
Define the Wasserstein distance of order $1$ by
\[
W_1(X,Y)
:=
\inf\Bigl\{\E |X'-Y'|:\ X'\stackrel{d}{=}X,\ Y'\stackrel{d}{=}Y\Bigr\},
\]
where the infimum is taken over all joint constructions $(X',Y')$ with the prescribed marginals
(see, for example, \cite[{\S}2.1.3]{Santambrogio2015OptimalTransport}; cf.\ also \cite{Villani2009OTON}).
\end{definition}

In the one-dimensional case it is convenient to use the formula in terms of distribution functions.

\begin{lemma}[One-dimensional formula for $W_1$ via the CDF]\label{lem:W1_cdf}
Let $X,Y\ge 0$ and $\E[X]+\E[Y]<\infty$. Then
\begin{equation}\label{eq:W1_cdf}
W_1(X,Y)=\int_0^\infty \bigl|F_X(t)-F_Y(t)\bigr|\,dt,
\end{equation}
where $F_X(t)=\PP(X\le t)$ and $F_Y(t)=\PP(Y\le t)$
(see, for example, \cite[{\S}2.1.3]{Santambrogio2015OptimalTransport}).
\end{lemma}

For a proof of this statement see, for example, \cite[{\S}2.1.3]{Santambrogio2015OptimalTransport}; this is the standard
derivation in one-dimensional optimal
transport, realised through the quantile (monotone) coupling.



\begin{proposition}[Exponential estimate in the metric $W_1$]\label{prop:w1_bound}
Assume that the conditions of Proposition~\ref{prop:wong_prop11_ru} are satisfied, and hence
that estimate \eqref{eq:cdf_exp_rate_recall} holds. Then for all $b>0$
\begin{equation}\label{eq:w1_bound}
\sup_{\theta\in[0,\theta^\ast]} W_1(R_b,R_\infty)\ \le\ \frac{C}{r}\,e^{-rb}.
\end{equation}
\end{proposition}

\begin{proof}
This is an immediate consequence of formula \eqref{eq:W1_cdf} and estimate \eqref{eq:cdf_exp_rate_recall}:
for fixed $\theta\in[0,\theta^\ast]$
\[
W_1(R_b,R_\infty)
=\int_0^\infty |F_{b,\theta}(t)-F_{\infty,\theta}(t)|\,dt
\le \int_0^\infty C e^{-r(b+t)}\,dt
=\frac{C}{r}e^{-rb}.
\]
Taking the supremum over $\theta\in[0,\theta^\ast]$, we obtain \eqref{eq:w1_bound}.
\end{proof}

\begin{corollary}[Quantile coupling with an exponentially small mean error]
\label{cor:explicit_coupling_exists}
Assume that the conditions of Proposition~\ref{prop:w1_bound} are satisfied. Fix $\theta\in[0,\theta^\ast]$ and $b>0$.
Let $U\sim\mathrm{Unif}(0,1)$ and $F^{-1}(u):=\inf\{t\ge0:\,F(t)\ge u\}$.
Define
\[
\widetilde R_b:=F_{b,\theta}^{-1}(U),\qquad
\widetilde R_\infty:=F_{\infty,\theta}^{-1}(U).
\]
Then $\widetilde R_b\stackrel{d}{=}R_b$, $\widetilde R_\infty\stackrel{d}{=}R_\infty$ and
\[
\E_\theta\bigl|\widetilde R_b-\widetilde R_\infty\bigr|
= W_1(R_b,R_\infty)
\le \frac{C}{r}e^{-rb}.
\]
\end{corollary}

\begin{proof}
The quantile coupling is optimal for one-dimensional $W_1$
(see, for example, \cite{Santambrogio2015OptimalTransport});
it remains to apply \eqref{eq:w1_bound}.
\end{proof}

\subsection{Examples of the use of overshoot bounds in applications}

\subsubsection{Error control for Lipschitz functionals}

The following corollary is standard: it is obtained from the
Kantorovich--Rubinstein inequality and the estimate of Proposition~\ref{prop:w1_bound}
(see, for example, \cite{Villani2009OTON}).

\begin{corollary}[Lipschitz functionals]
\label{cor:lipschitz_application}
Assume that the conditions of Proposition~\ref{prop:w1_bound} are satisfied.
Then for every Lipschitz function $g:\R_+\to\R$ we have for all $b>0$
\[
\sup_{\theta\in[0,\theta^\ast]}
\bigl|\E_\theta g(R_b)-\E_\theta g(R_\infty)\bigr|
\le
\mathrm{Lip}(g)\,
\sup_{\theta\in[0,\theta^\ast]} W_1(R_b,R_\infty)
\le
\mathrm{Lip}(g)\,\frac{C}{r}e^{-rb}.
\]
\end{corollary}

\begin{remark}[Examples]
1)~
 $g(y)=(y-a)^+$: $\mathrm{Lip}(g)=1$; ~~
2)~ $g(y)=e^{-\lambda y}$: $\mathrm{Lip}(g)=\lambda$.
\end{remark}

\subsubsection{Smoothed total variation (fixed smoothing)}

\begin{definition}[Total variation]
For probability measures $\mu,\nu$ on $\R$ set
\[
\|\mu-\nu\|_{\mathrm{TV}}:=\sup_{A\subset\R}|\mu(A)-\nu(A)|.
\]
\end{definition}

\begin{definition}[Smoothing]
Let $U\sim\mathrm{Unif}(0,1)$ be independent of $X$. Set
\[
X^{\mathrm{sm}}:=X+U .
\]
\end{definition}

\begin{lemma}[Smoothing converts $W_1$ into TV
]
\label{lem:TV_smoothed_vs_W1_noeps}
For any random variables $X,Y$ with $\E|X|+\E|Y|<\infty$,
\begin{equation}\label{eq:TV_estimation_via_W1_noeps}
\bigl\|\Law(X^{\mathrm{sm}})-\Law(Y^{\mathrm{sm}})\bigr\|_{\mathrm{TV}}
\ \le\ W_1(X,Y).
\end{equation}
holds.
\end{lemma}

\begin{proof}
Fix a measurable set $A\subset\R$ and define the test function
\[
f_A(x):=\PP(x+U\in A)=\int_0^1 \mathbf 1_A(x+u)\,du.
\]
Then $0\le f_A\le 1$. We show that $f_A$ is $1$-Lipschitz.
For $h>0$ we have
\[
\begin{aligned}
f_A(x+h)-f_A(x)
&=\int_0^1 \mathbf 1_A(x+h+u)\,du-\int_0^1 \mathbf 1_A(x+u)\,du\\
&=\int_h^{1+h}\mathbf 1_A(x+v)\,dv-\int_0^1 \mathbf 1_A(x+v)\,dv\\
&=\int_{1}^{1+h}\mathbf 1_A(x+v)\,dv-\int_0^{h}\mathbf 1_A(x+v)\,dv,
\end{aligned}
\]
whence $|f_A(x+h)-f_A(x)|\le h$. Similarly, for $h<0$ we obtain
$|f_A(x+h)-f_A(x)|\le |h|$. Consequently, $\mathrm{Lip}(f_A)\le 1$.

Now for any joint construction $(X',Y')$ with $X'\stackrel d= X$, $Y'\stackrel d=Y$ we have
\[
\bigl|\E f_A(X')-\E f_A(Y')\bigr|\le \E|X'-Y'|.
\]
Taking the infimum over all couplings, we obtain
\[
\bigl|\E f_A(X)-\E f_A(Y)\bigr|\le W_1(X,Y).
\]
It remains to take the supremum over $A$, which yields \eqref{eq:TV_estimation_via_W1_noeps}.
\end{proof}

\begin{corollary}[TV bound for smoothed overshoots]\label{cor:TV_smoothed_rate_noeps}
Assume that the conditions of Proposition~\ref{prop:w1_bound} are satisfied.
Define $R_b^{\mathrm{sm}}:=R_b+U$ and $R_\infty^{\mathrm{sm}}:=R_\infty+U$,
where $U\sim\mathrm{Unif}(0,1)$ is independent of all the remaining random quantities.
Then for all $b>0$
\[
\sup_{\theta\in[0,\theta^\ast]}
\bigl\|\Law(R_b^{\mathrm{sm}})-\Law(R_\infty^{\mathrm{sm}})\bigr\|_{\mathrm{TV}}
\le
\sup_{\theta\in[0,\theta^\ast]} W_1(R_b,R_\infty)
\le
\frac{C}{r}\,e^{-rb}.
\]
\end{corollary}

\begin{remark}[Saturation at level $1$ and sharpness of the constant]
Since always $\|\mu-\nu\|_{\mathrm{TV}}\le 1$, one also has
\[
\bigl\|\Law(X^{\mathrm{sm}})-\Law(Y^{\mathrm{sm}})\bigr\|_{\mathrm{TV}}
\le
\min\{1,\ W_1(X,Y)\}.
\]
\end{remark}


\subsubsection{Threshold stopping: explicit approximation error}
\label{subsec:example_threshold_policy}

Consider $\tau(b)=\inf\{n\ge1:\,S_n>b\}$ and the overshoot $R_b=S_{\tau(b)}-b$.
In many problems of threshold signalling and regenerative models, the following are important:
(i) the mean time until crossing the level, $\E_\theta[\tau(b)]$, and
(ii) the penalties $\E_\theta[g(R_b)]$ for exceeding the threshold, where $g:\R_+\to\R$ is Lipschitz.

\textbf{Exact formula and the role of the overshoot.}
In our model, for $\theta>0$ we have $\E_\theta|X_1|<\infty$ and $\E_\theta[\tau(b)]<\infty$.
For example, in the standard exponential family
$\E_\theta[e^{-\theta X_1}]=e^{-\psi(\theta)}<1$, and on the event $\{\tau(b)>n\}$ we have $S_n\le b$,
therefore
\[
\P_\theta(\tau(b)>n)\le e^{\theta b}\,\E_\theta[e^{-\theta S_n}]
= e^{\theta b}\bigl(\E_\theta[e^{-\theta X_1}]\bigr)^n
= e^{\theta b}e^{-n\psi(\theta)},
\]
that is, the tail of $\tau(b)$ decays geometrically and $\E_\theta[\tau(b)]<\infty$.
Consequently, Wald's identity is applicable, and
\begin{equation}\label{eq:tau_mean_exact}
\E_\theta[\tau(b)]
=\frac{\E_\theta[S_{\tau(b)}]}{\mu_\theta}
=\frac{b+\E_\theta[R_b]}{\mu_\theta}.
\end{equation}
Thus, control of $\E_\theta[R_b]$ directly determines the accuracy of approximations for $\E_\theta[\tau(b)]$.

\textbf{Exponential error.}
Set $\kappa_\theta:=\E_\theta[R_\infty]$ (in particular, $\kappa_\theta<\infty$).
From Corollary~\ref{cor:lipschitz_application}, with $g(y)=y$ (where $\mathrm{Lip}(g)=1$), we obtain:
for all $b>0$ and $\theta\in(0,\theta^\ast]$
\[
\bigl|\E_\theta[R_b]-\kappa_\theta\bigr|
\le \frac{C}{r}e^{-rb}.
\]
Substituting this into \eqref{eq:tau_mean_exact}, we obtain the explicit error bound
\begin{equation}\label{eq:tau_mean_error}
\left|\E_\theta[\tau(b)]-\frac{b+\kappa_\theta}{\mu_\theta}\right|
\le \frac{C}{r\,\mu_\theta}e^{-rb}.
\end{equation}
Thus, the correction $b\mapsto b+\kappa_\theta$ (continuity correction) has an error
that is exponentially small in $b$ for fixed $\theta$.


\section{Conclusion}
\label{sec:conclusion}

In this paper we obtain Lorden-type bounds, uniform in the level $b$, for overshoot moments
$R_b=S_{\tau(b)}-b$ of a random walk with sign-changing increments and positive {\color{blue}drift}
within the model of a standard exponential family. For any $k\in\mathbb N$ it is shown that there exists
an exponentially small correction in $b$ such that for $\theta\in(0,\theta^\ast]$ and $b>0$
\[
\E_\theta\!\left[R_b^{\,k}\right]
\ \le\
\frac{1}{k+1}\,\frac{\E_\theta[(X_1^+)^{k+1}]}{\mu_\theta}
\ +\ O(e^{-rb}).
\]
It follows that the classical constant in the moment bound improves to $C_k=1$:
for fixed $\theta$ this is true for all sufficiently large $b$, and in the small-drift regime
(as $\theta\downarrow 0$) it is uniform over all $b\ge 0$.

The proof relies on the reduction to the renewal process of strict ascending ladder heights
and on the uniform exponential rate of convergence of the distribution of $R_b$ to the limiting $R_\infty$.
The Appendix contains counterexamples showing that strengthening the bound to the form with denominator
$k\mu_\theta$ (for $k>1$) cannot hold simultaneously for all $(b,\theta)$ even within the class of
standard exponential families.

Finally, Section~\ref{sec:coupling_viewpoint} extracts from the exponential CDF estimate the standard,
but application-friendly, consequences: exponential convergence in the Wasserstein distance $W_1$,
a quantile coupling with $\E|\widetilde R_b-\widetilde R_\infty|=O(e^{-rb})$, error bounds for
Lipschitz functionals, and control in total variation after smoothing.


\appendix
\section{Counterexamples: why the strengthened bound with $k\mu_\theta$ in the denominator is false in general: two counterexamples}
\label{app:counterexamples}

In this appendix we show that for $k>1$ one cannot expect the validity of the
\emph{strengthened} moment bound of the form
\begin{equation}\label{eq:target-ru}
\E_\theta\!\left[R_b^{\,k}\right]
\ \le\
\frac{\E_\theta\!\left[(X_1^+)^{k+1}\right]}{k\,\mu_\theta},
\qquad b\ge 0,
\end{equation}
\emph{simultaneously} for all levels $b$ and, in parametric models, for all parameter values.
Here $R_b=S_{\tau(b)}-b$, $\tau(b)=\inf\{n\ge1:\,S_n>b\}$, $\mu_\theta=\E_\theta[X_1]>0$,
$x^+=\max\{x,0\}$.

\subsection{Counterexample 1: deterministic increments}

\begin{proposition}[Deterministic counterexample]\label{prop:cx_deterministic}
Let $X_1\equiv c>0$ almost surely. Then for every $k>1$
inequality \eqref{eq:target-ru} already fails at $b=0$.
\end{proposition}

\begin{proof}
For $b=0$ we have $\tau(0)=1$ and $R_0=S_{\tau(0)}=X_1=c$, hence
\[
\E[R_0^{\,k}]=c^k.
\]
On the other hand, $X_1^+=c$, $\mu=\E[X_1]=c$, and the right-hand side of \eqref{eq:target-ru} is equal to
\[
\frac{\E[(X_1^+)^{k+1}]}{k\,\mu}=\frac{c^{k+1}}{k\,c}=\frac{c^k}{k}.
\]
Since $k>1$, we have $c^k>\frac{c^k}{k}$, and therefore \eqref{eq:target-ru} is false.
\end{proof}

\begin{remark}
If one wishes formally to have $b>0$ rather than $b=0$, it suffices to take any
$b\in\bigl(0,\;c(1-k^{-1/k})\bigr)$: then $\tau(b)=1$, $R_b=c-b$ and
$(c-b)^k>c^k/k$.
\end{remark}

\subsection{Counterexample 2: a standardised exponential family based on the uniform distribution}

\begin{proposition}[Counterexample within the standard exponential family]
\label{prop:cx_uniform_tilt}
Set $a:=\sqrt{3}$ and let $F_0=\mathrm{Unif}[-a,a]$.
Consider the standard exponential family $\{F_\theta:\theta\ge0\}$ defined by tilting
\[
F_\theta(\mathrm{d}x)=e^{\theta x-\psi(\theta)}\,F_0(\mathrm{d}x),
\qquad
\psi(\theta):=\log \E_0[e^{\theta X_1}].
\]
Then for every $k>1$ there exist $b\in(0,a)$ and $\theta_0>0$ such that for all
$\theta\ge\theta_0$ inequality \eqref{eq:target-ru} fails.
\end{proposition}

\begin{proof}
\textbf{Step 1: explicit form of $F_\theta$ and concentration as $\theta\to\infty$.}
Since $F_0$ is uniform on $[-a,a]$, for $\theta>0$ the distribution $F_\theta$ has density
\[
f_\theta(x)
=
\frac{e^{\theta x}}{\int_{-a}^{a}e^{\theta u}\,\mathrm{d}u}\,\mathbf{1}_{[-a,a]}(x)
=
\frac{\theta\,e^{\theta x}}{e^{a\theta}-e^{-a\theta}}\,\mathbf{1}_{[-a,a]}(x).
\]
For any $\varepsilon\in(0,2a)$ we obtain
\[
\PP_\theta(X_1\le a-\varepsilon)
=
\frac{\int_{-a}^{a-\varepsilon} e^{\theta x}\,\mathrm{d}x}{\int_{-a}^{a} e^{\theta x}\,\mathrm{d}x}
=
\frac{e^{\theta(a-\varepsilon)}-e^{-a\theta}}{e^{a\theta}-e^{-a\theta}}
\le e^{-\theta\varepsilon}\ \xrightarrow[\theta\to\infty]{}\ 0.
\]
Hence $X_1\to a$ in probability as $\theta\to\infty$.
Since $|X_1|\le a$ almost surely, for every $m>0$ we also have
$X_1\to a$ in $L^m$, and hence
\begin{equation}\label{eq:limits_uniform_tilt}
\mu_\theta=\E_\theta[X_1]\to a,
\qquad
\E_\theta[(X_1^+)^{k+1}]\to a^{k+1},
\qquad \theta\to\infty.
\end{equation}

\textbf{Step 2: choice of the barrier $b$.}
Set $t:=k^{-1/k}\in(0,1)$ and choose
\[
b:=a\,\frac{1-t}{2}\in(0,a).
\]
Then
\[
\frac{a-b}{a}=\frac{1+t}{2}>t
\quad\Longrightarrow\quad
(a-b)^k=a^k\Big(\frac{1+t}{2}\Big)^k>a^k t^k=\frac{a^k}{k}.
\]
Thus, for this choice of $b$,
\begin{equation}\label{eq:choose_b_uniform}
(a-b)^k>\frac{a^k}{k}.
\end{equation}

\textbf{Step 3: asymptotics of $\E_\theta[R_b^k]$ as $\theta\to\infty$.}
Fix the chosen $b\in(0,a)$. Then
\[
\PP_\theta(\tau(b)=1)=\PP_\theta(X_1>b)\xrightarrow[\theta\to\infty]{}1,
\]
since $X_1\to a>b$ in probability.
On the event $\{\tau(b)=1\}$ we have $R_b=X_1-b$.
Moreover, since $X_1\le a$ almost surely and $S_{\tau(b)-1}\le b$, we have
$0\le R_b=S_{\tau(b)}-b\le a$ almost surely, hence $0\le R_b^k\le a^k$.

Decompose the expectation:
\[
\E_\theta[R_b^k]
=
\E_\theta\big[(X_1-b)^k;\,X_1>b\big]
+
\E_\theta\big[R_b^k;\,X_1\le b\big].
\]
The second term is estimated as
\[
0\le \E_\theta\big[R_b^k;\,X_1\le b\big]\le a^k\,\PP_\theta(X_1\le b)\xrightarrow[\theta\to\infty]{}0.
\]
For the first term we use the $L^k$ convergence $X_1\to a$:
\[
\E_\theta\big[(X_1-b)^k;\,X_1>b\big]
=
\E_\theta[(X_1-b)^k]-\E_\theta\big[(X_1-b)^k;\,X_1\le b\big]
\longrightarrow
(a-b)^k-0=(a-b)^k.
\]
Thus,
\begin{equation}\label{eq:limit_left_uniform}
\E_\theta[R_b^k]\ \longrightarrow\ (a-b)^k,
\qquad \theta\to\infty.
\end{equation}

\textbf{Step 4: comparison with the right-hand side of \eqref{eq:target-ru}.}
From \eqref{eq:limits_uniform_tilt} we obtain
\[
\frac{\E_\theta[(X_1^+)^{k+1}]}{k\,\mu_\theta}
\ \longrightarrow\
\frac{a^{k+1}}{k\,a}
=
\frac{a^k}{k},
\qquad \theta\to\infty.
\]
Combining this with \eqref{eq:limit_left_uniform} and \eqref{eq:choose_b_uniform}, we obtain:
for all sufficiently large $\theta$,
\[
\E_\theta[R_b^k]>\frac{\E_\theta[(X_1^+)^{k+1}]}{k\,\mu_\theta},
\]
that is, \eqref{eq:target-ru} fails.
\end{proof}

\begin{remark}
Counterexample~\ref{prop:cx_uniform_tilt} lies within the class of standard exponential families:
$F_0=\mathrm{Unif}[-\sqrt{3},\sqrt{3}]$ has $\E_0[X_1]=0$, $\Var_0(X_1)=1$, is absolutely continuous,
and therefore strongly non-lattice. Consequently, the failure of \eqref{eq:target-ru} is not a consequence of lattice effects or of the absence of exponential moments.
\end{remark}


\textbf{Acknowledgement.}~
The work of the second author was carried out within the framework of the
Fundamental Research Programme of HSE University.

\bibliographystyle{plain} 
\bibliography{references} 

\end{document}